\documentclass[11pt,a4paper]{article}
\usepackage[utf8]{inputenc}
\usepackage[T1]{fontenc}
\usepackage{lmodern}
\usepackage{amsmath,amssymb,amsthm,mathtools}
\usepackage[margin=3.5cm]{geometry}
\usepackage{xcolor}
\usepackage{hyperref}

\newtheorem{theorem}{Theorem}[section]
\newtheorem{proposition}[theorem]{Proposition}
\newtheorem{lemma}[theorem]{Lemma}
\newtheorem{corollary}[theorem]{Corollary}
\theoremstyle{definition}
\newtheorem{definition}[theorem]{Definition}
\theoremstyle{remark}
\newtheorem{remark}[theorem]{Remark}

\newcommand{\Lc}{\mathring{\mathcal L}}
\newcommand{\Ric}{\operatorname{Ric}}
\newcommand{\tr}{\operatorname{tr}}
\newcommand{\dv}{\operatorname{div}}
\newcommand{\Rm}{\mathring{R}}
\newcommand{\supp}{\operatorname{supp}}
\newcommand{\nv}[1]{{#1}}
\newcommand{\Hs}{\mathcal H}

\title{Compactly supported instability of Riemannian Schwarzschild–Tangherlini metrics}
\author{Erwann Delay}
\date{}

\begin{document}
\maketitle

\begin{abstract}
On the Riemannian Schwarzschild--Tangherlini metric of dimension $n\ge4$, we show that the
destabilizing TT tensor of Ozuch lies in the image of the ``TT-tensor machine''. By continuity,
it can then be approximated, in suitable weighted Sobolev spaces, by smooth compactly supported
TT tensors, on which the stability quadratic form is still negative. This proves instability in
the usual sense for noncompact manifolds. To our knowledge, no exact smooth compactly supported
destabilizing TT tensor was known before, even in 4D.
\end{abstract}

\tableofcontents

%====================================================================
\section{Introduction}
%====================================================================

Let $(M,g)$ be a complete Ricci-flat manifold, let $\Delta_L=\nabla^*\nabla -2\Rm$ be the
Lichnerowicz Laplacian and let
$$\mathcal Q(u)=\int_M\langle \Delta_L u,u\rangle\,d\mu_g$$ be the stability quadratic
form (up to a factor, the second variation of the total scalar curvature on TT tensors). On a
compact manifold, $g$ is \emph{linearly stable} if $\mathcal Q\ge0$ on TT tensors, and
\emph{unstable} otherwise. On a noncompact manifold, the usual definition (see for instance Dahl
and Kröncke \cite{DK24}) only involves compactly supported tensors:
\begin{definition}\label{def:usual}
The metric $g$ is \emph{stable} if $\mathcal Q(u)\ge0$ for every smooth, compactly supported TT
tensor $u$, and \emph{unstable} otherwise, that is, if there exists a smooth, compactly supported
TT tensor $u$ with $\mathcal Q(u)<0$.
\end{definition}
This definition is the natural one from the
variational point of view, since compactly supported deformations do not change the
geometry of the end.

The recent papers \cite{BiqOz26,Del26,Ozu26} use a geometrically natural but weaker
notion: $g$ is unstable if there exists a smooth TT tensor $u$, not necessarily compactly
supported but with appropriate asymptotic behavior (see \cite[Remark~13]{BiqOz26} and
\cite[Remark~3.1]{Del26}), such that $\mathcal Q(u)$ is finite and negative. The two notions are
not obviously equivalent. Indeed, if $u$ is TT with $\mathcal Q(u)<0$, the natural candidates
$\eta_Ru$, where $\eta_R$ is a cut-off function, are no longer divergence-free, and projecting
them back onto TT tensors requires solving an elliptic equation on $M$, whose solution has, in
general, no reason to be compactly supported. Passing from this geometric notion to
Definition~\ref{def:usual} therefore amounts to a density statement for compactly supported TT
tensors in a weighted energy space on which $\mathcal Q$ is continuous, and such a statement is
not available in general.

\nv{Another possible route is given by localized gluing techniques of Corvino type
\cite{Cor00,CD03,Del12}: one could try to correct $\dv(\eta_Ru)$ by a trace-free tensor $v$
supported in the annulus $A_R=\{R\le r\le2R\}$, solving the underdetermined equation
$\dv v=-\dv(\eta_Ru)$. This equation can be solved modulo a finite-dimensional cokernel, made of
the conformal Killing fields of the annulus. The Killing fields of $M$ (here $\partial_\tau$ and the
rotations) create no obstruction since $u$ is TT. However, the end is asymptotic to the flat
$S^1\times\mathbb R^N$, so that the annulus carries approximate conformal Killing fields (coming
for instance from the translations and dilations of $\mathbb R^N$). Their contribution, as well as
the norm of the correction in the energy space, would have to be controlled as $R\to\infty$; we do
not pursue this here. The approach followed in this paper avoids these issues: it involves no
obstruction and produces an exact, explicit formula for the compactly supported TT tensor.}

The Riemannian Schwarzschild metric is known to be linearly unstable: Gross, Perry and Yaffe
exhibited numerically a negative mode of the Lichnerowicz Laplacian on TT tensors \cite{GPY82}.
This eigentensor lies in $L^2$ and decays exponentially, but it is \emph{not} compactly supported
either, so it is not by itself an unstable direction in the sense of Definition~\ref{def:usual};
the same cut-off problem arises. In every dimension $n\ge4$, the Riemannian
Schwarzschild--Tangherlini metrics are unstable as well (see \cite{BiqOz26,Del26} for small
dimensions and \cite{Ozu26} for all dimensions\nv{; see also \cite{Pre00,Rea01} for Euclidean
negative modes in higher dimensions and their relation with the Gregory--Laflamme instability
\cite{GL93}}). More precisely, Ozuch \cite{Ozu26} gave an
explicit TT tensor $k$, built from the Killing field $\partial_\tau$, such that $\mathcal Q(k)$ is
finite and negative. This tensor decays only polynomially, it is not compactly supported, and
multiplying it by a cut-off function destroys the TT property.

The main result of this article is the following.
\begin{theorem}
For every $n\ge4$, there exists an \emph{exact, smooth, compactly supported} TT tensor with
negative $\mathcal Q$ on the Riemannian Schwarzschild--Tangherlini metric (see
Theorem~\ref{thm:main}). Hence these metrics are unstable in the usual sense of
Definition~\ref{def:usual}.
\end{theorem}
In particular, these metrics are unstable not only in the geometric sense of \cite{Ozu26}.
Since neither the negative mode of \cite{GPY82} nor the tensor of \cite{Ozu26} is compactly
supported, to our knowledge no exact smooth compactly supported destabilizing TT tensor was
known before, even in 4D.

The construction is based on the ``TT-tensor machine'' \cite{DelTT}, which provides, on any
Einstein manifold, a fourth-order operator $P$ mapping trace-free symmetric tensors to TT
tensors. If one finds $H$ with $P(H)=k$, then $P(\eta_RH)$ is TT, smooth and compactly supported,
and it remains to check that it converges to $k$ in a topology for which the quadratic form is
continuous. This paper carries out this program explicitly.

More precisely, we construct explicit, smooth, trace-free symmetric $2$-tensors $h$ and $H$ such
that
\[
\Delta_L h=k+\Lc X,\qquad \Delta_L H=h+\Lc\tilde X,
\]
where $\Lc$ is the conformal Killing operator. The method replaces $\Delta_L$ by the linearized
Ricci operator, which is legitimate modulo $\operatorname{Im}\Lc$. This turns the problem into a
triangular system of radial ODEs: a first-order constraint, solved by quadrature, and a scalar
second-order ODE whose homogeneous solutions are known. We deduce that $P(H)=k$. Working in
weighted Sobolev spaces, we then show that, for $R$ large, the tensor $P(\eta_RH)$, where
$\eta_R$ is a radial cut-off function, is an exact, smooth, compactly supported TT tensor on
which the stability quadratic form is negative. This proves instability in the usual sense for
noncompact manifolds.

\paragraph{Plan.} Section~\ref{sec:setting} fixes the setting. Section~\ref{sec:method} explains
the method, and Sections~\ref{sec:h} and~\ref{sec:H} solve the two successive equations.
Section~\ref{sec:PH} proves that $P(H)=k$.
Section~\ref{sec:sobolev} introduces weighted Sobolev spaces and establishes the required
continuity and convergence properties, and Section~\ref{sec:main} proves the main theorem.
Section~\ref{sec:persp} discusses, at a formal level, an abstract existence argument (Fredholm
alternative), the extension to Kerr and Myers--Perry metrics, and the use of $P$ as a numerical
tool to detect instability.

%====================================================================
\section{Setting, conventions and statement of the problem}\label{sec:setting}
%====================================================================

\subsection{The metric and the normalization}
Let $n\ge4$, $p=n-3$, $d=n-2$, and let \nv{$h_Y$ be the round metric of the unit sphere
$Y=S^{n-2}$, so that $\Ric_{h_Y}=(n-3)h_Y$}. The Riemannian
Schwarzschild--Tangherlini metric is
\[
g=f\,d\tau^2+f^{-1}dr^2+r^2h_Y,\qquad f(r)=1-\Big(\frac{r_0}{r}\Big)^p,
\]
with $\tau$ periodic of period $4\pi r_0/p$. It is smooth and Ricci-flat on $M=\{r\ge r_0\}$, the
\nv{sphere $\{r=r_0\}\simeq S^{n-2}$} being a bolt.

By homothety, we may take $r_0=1$. Indeed, denote by $g_{r_0}$ the metric
above. The change of variables $(\tau,r)=r_0(\tilde\tau,\tilde r)$ gives
$g_{r_0}=r_0^2\big(f_1(\tilde r)\,d\tilde\tau^2+f_1(\tilde r)^{-1}d\tilde r^2+\tilde r^2h_Y\big)$,
that is, $g_{r_0}$ is isometric to $r_0^2\,g_1$. Moreover, for a constant $c>0$, the Levi-Civita
connection of $c^2g$ is that of $g$, so that, in orthonormal frames,
$\Delta_L^{c^2g}=c^{-2}\Delta_L^{g}$, and similarly for $\Lc$. With the normalization of
Section~\ref{subsec:source}, the frame components of $k$ (in the orthonormal frame
$e_\tau,e_r,r^{-1}\epsilon_i$, with $\epsilon_i$ orthonormal for $h_Y$) are functions of $r/r_0$
only. Since $\Delta_Lh=k+\Lc X$ and $\Delta_LH=h+\Lc\tilde X$, the frame components of $h$ and
$H$ for a general $r_0$ are obtained from those computed for $r_0=1$, denoted $h_1$ and $H_1$, as
\[
h_{r_0}(r)=r_0^{2}\,h_1(r/r_0),\qquad H_{r_0}(r)=r_0^{4}\,H_1(r/r_0).
\]
For instance, for $n=4$ we find $h_1=\frac3{16}\big(g_1-2r^2h_Y\big)$ in Section~\ref{sec:H}; this
gives $h_{r_0}=\frac{3r_0^2}{16}\big(g_{r_0}-2r^2h_Y\big)$, that is,
$h=\frac{3m^2}4\big(g-2r^2g_{S^2}\big)$ for $r_0=2m$, in accordance with Theorem~\ref{thm:h4}.
Similarly, the tensor $H_1=\frac18\big(2r-1-\frac{r^3\ln r}{r-1}\big)(P_0-P_1)$ of
Section~\ref{sec:H} gives
$H_{r_0}=\frac{r_0^4}8\big(2\frac r{r_0}-1-\frac{(r/r_0)^3\ln(r/r_0)}{r/r_0-1}\big)(P_0-P_1)$.
Note that smoothness at the bolt is preserved by the rescaling.
From now on, $r_0=1$ unless otherwise stated.

We will write
\[
D(r)=r^p-1,\qquad f=\frac{D}{r^p},\qquad f'=p\,r^{-p-1},\qquad r^df'=p,
\]
the last identity being used constantly. The basic invariant tensors are
\[
P_0=f\,d\tau^2,\qquad P_1=f^{-1}dr^2,\qquad P_2=r^2h_Y\quad(\text{of rank }d),
\]
so that $g=P_0+P_1+P_2$. An invariant diagonal tensor reads $aP_0+bP_1+cP_2$, its trace is
$a+b+dc$ and its trace-free part is $[T]_0=T-\frac{\tr T}{n}g$. We denote by $e_\tau,e_r$ the unit
vectors along $\partial_\tau,\partial_r$ (so that $P_0=e_\tau^2$ and $P_1=e_r^2$), and call
$a,b,c$ the frame components.

\subsection{Conventions}
On functions $\Delta=-\tr\nabla^2=d^*d$, with $d^*\omega=-\nabla^i\omega_i$ on $1$-forms. The
divergence of a symmetric $2$-tensor is
\[
(\dv h)_i=-\nabla^jh_{ij}.
\]
Vector fields and $1$-forms are identified through $g$. We write
$\mathcal L\omega=\nabla\omega+(\nabla\omega)^T$ (so that $\mathcal LX^\flat=\mathcal L_Xg$),
$D\Ric(g)$ for the linearization of $\Ric$ at $g$, and $B=2\dv+d\tr$ for the Bianchi operator.
\emph{The prime $'$ denotes the derivative with respect to $r$.} On a Ricci-flat manifold
the Lichnerowicz Laplacian is
\[
\Delta_L=\nabla^*\nabla-2\Rm,\qquad (\Rm h)_{ij}=R_{ikjl}h^{kl},
\]
with the sign convention $\Rm h=(\tr h)g-h$ on the unit sphere. The conformal Killing operator is
\[
\Lc X=\mathcal L_Xg+\frac2n(d^*X)\,g .
\]
The stability quadratic form is
\[
\mathcal Q(u)=\int_M\big(|\nabla u|^2-2\langle\Rm u,u\rangle\big)\,d\mu_g=\langle \Delta_Lu,u\rangle_{L^2}
\quad\text{for compactly supported }u .
\]

\subsection{The source tensor}\label{subsec:source}
Let $K=\partial_\tau$, $\kappa=K^\flat=f\,d\tau$, $|K|^2=f$, $F=d\kappa$. With
\[
a_n=\frac{n+2}{2(n-1)},\qquad b_n=\frac3{2(n-1)},\qquad \lambda_n=n-1-a_n,
\]
let $u$ be the smooth radial solution of $\Delta u=a_n f+\lambda_n$, given by
\[
r^df\,u'=-r^{d+1}+\frac{a_n}2r^2+1-\frac{a_n}2,
\]
and let $k$ be Ozuch's TT tensor \cite{Ozu26}
\[
k=\nabla^2u-\kappa\otimes\kappa+\Big(b_nf+\frac{\lambda_n}n\Big)g .
\]
It satisfies $\Delta_Lk=(F\circ F)_0-\frac{n-4}{2(n-1)}(\nabla^2f)_0$. Its part along $Y$ is
$\frac{n-4}{4(n-1)}r^{-p}+O(r^{-p-2})$, so $k=O(r^{3-n})$ for $n\ge5$. For $n=4$ (with $r_0=2m$)
\nv{the frame components of $k$ are}
\[
k_{\tau\tau}=-\frac{m^2}{r^2}-\frac{2m^3}{r^3},\qquad k_{rr}=\frac{m^2}{r^2}-\frac{6m^3}{r^3},\qquad
k|_{S^2}=\frac{4m^3}{r^3}\,\nv{P_2},
\]
hence $k=O(r^{-2})$.

\begin{theorem}[Ozuch \cite{Ozu26}]\label{thm:Ozuch}
For every $n\ge4$, $\mathcal Q(k)$ is finite and negative.
\end{theorem}

\begin{proof}[Comments]
For $n\ge5$ this is \cite[Section 3.4 and (41)]{Ozu26}:
$\mathcal Q(k)=-\frac{3(n+2)}{4(n-1)}\int|df|^2-\frac{n-5}2\int|F|^2<0$. For $n=4$,
\cite[Remark 3.6]{Ozu26} gives $\mathcal Q(k)=-\frac12\int(|df|^2-|d\varpi|^2)$, where
$d\varpi=\iota_K({*F})$ \nv{(this function is denoted $\chi$ in \cite{Ozu26})}. Since $K$ is
hypersurface orthogonal, $\kappa\wedge d\kappa=0$, so $d\varpi=0$ and
$\mathcal Q(k)=-\frac12\int|df|^2<0$.

\nv{Here $\mathcal Q(k)$ is the energy expression
$\int_M\big(|\nabla k|^2-2\langle\Rm k,k\rangle\big)\,d\mu_g$, which is the one used in
Section~\ref{sec:main}. Indeed the formulas above are obtained by integrations by parts on
$\{r\le R\}$ and letting $R\to\infty$; since $k=O(r^{a_k})$ and $\nabla k=O(r^{a_k-1})$, with
$a_k=3-n$ for $n\ge5$ and $a_k=-2$ for $n=4$, every boundary term on $\{r=R\}$ is
$O(R^{2a_k+n-3})$, that is $O(R^{3-n})$ for $n\ge5$ and $O(R^{-3})$ for $n=4$, and tends to $0$.
This holds in particular for $n=5$, where $k\notin L^2$.}
\end{proof}

\subsection{The problem}
We want to find $h$, then $H$, symmetric, trace-free and smooth (including at the bolt), such that
\[
\Delta_Lh-k\in\operatorname{Im}\Lc,\qquad \Delta_LH-h\in\operatorname{Im}\Lc,
\]
where the vector fields involved are also smooth, and where each solution grows at most two powers
of $r$ faster than its source (possibly up to logarithmic factors).
In order to reduce the problem to ODEs, we also look for solutions invariant under
\nv{$U(1)\times O(n-1)$ and under the reflection $\tau\mapsto-\tau$}.

%====================================================================
\section{Method}\label{sec:method}
%====================================================================

\subsection{Three identities on a Ricci-flat manifold}
We first recall some basic facts (see e.g.\ \cite{DelTT}).
\begin{lemma}\label{lem:comm}
On a Ricci-flat manifold:
\begin{enumerate}
\item $\Delta_L(\varphi g)=(\Delta\varphi)\,g$; in particular $\Delta_L$ commutes with the
trace-free projection: $\Delta_L([T]_0)=[\Delta_LT]_0$;
\item $\Delta_L(\mathcal L_Zg)=\mathcal L_{\nabla^*\nabla Z}\,g$ and $\Delta_L\,\Lc=\Lc\,\Delta_H$;
\item $\nabla^2\varphi=\frac12\mathcal L_{\nabla\varphi}g$.
\end{enumerate}
\end{lemma}

\begin{lemma}[Lichnerowicz Laplacian and linearized Ricci]\label{lem:ricci}
For every symmetric $2$-tensor $h$ on a Ricci-flat manifold,
\[
\Delta_L h=2\,D\Ric(g)h+\tfrac12\,\mathcal L\big(B(h)\big),\qquad B(h)=2\dv h+d\tr h .
\]
\end{lemma}

\begin{proof}
This is the classical formula $D\Ric(g)=\frac12\Delta_L-\frac14\mathcal LB$ \cite{Bes87}.
\end{proof}

\begin{proposition}[Principle of the method]\label{prop:principle}
Let $S$ be a symmetric $2$-tensor with $S\equiv k$ modulo $\operatorname{Im}\mathcal L+C^\infty(M)g$.
If there are $\bar h$, a function $\varphi$ and a vector field $\xi$ with
\[
2\,D\Ric(g)\bar h=S+\varphi g+\mathcal L_\xi g,
\]
then $h:=[\bar h]_0$ satisfies $\Delta_L h=k+\Lc X$ for an explicit vector field $X$.
\end{proposition}

\begin{proof}
By Lemma~\ref{lem:ricci}, $\Delta_L\bar h=S+\varphi g+\mathcal L_\xi g+\frac12\mathcal LB(\bar h)$.
Take the trace-free part using Lemma~\ref{lem:comm}, and replace $S$ by $k$ modulo
$\operatorname{Im}\Lc$.
\end{proof}

\begin{remark}
Let us explain why this is simpler than inverting $\Delta_L$ directly.
\begin{itemize}
\item $D\Ric(g)(\mathcal L_Zg)=0$: a gauge condition can be imposed on the unknown $\bar h$ at no
cost (for instance: no $P_2$ component), and the vector field $\xi$ in the right-hand side can be
chosen freely.
\item The source can be simplified before any computation is done: since
$\nabla^2u=\frac12\mathcal L_{\nabla u}g$,
\[
k\equiv-\kappa\otimes\kappa=-f\,P_0\qquad\text{modulo }\operatorname{Im}\mathcal L+C^\infty g,
\]
and the function $u$ disappears.
\item $\bar h=2\alpha P_0+2\beta P_1$ is the variation of a family of static metrics
$N^2d\tau^2+e^{2\Lambda}dr^2+r^2h_Y$ (with $N=\sqrt f(1+\alpha)$, $e^{2\Lambda}=f^{-1}(1+2\beta)$),
whose Ricci tensor is known in closed form, so that linearizing amounts to differentiating scalar
expressions.
\item The linearized Einstein equations contain a first-order constraint, which makes the system
triangular.
\item The homogeneous solutions (modulo gauge) are the geometric deformations of the family
(rescaling of $\tau$, variation of the mass, growing mode), which makes it possible to use
variation of constants.
\end{itemize}
\end{remark}

\subsection{Ricci curvature of static metrics with \nv{$U(1)\times O(n-1)$} symmetry}
The following formulas are the classical expressions of the Ricci curvature of a warped
product (see \cite[Corollary 7.43]{ONe83} or \cite[Proposition 9.106]{Bes87}) applied to the static
metric $g=N^2d\tau^2+\hat g$; see also \cite{And00} for static vacuum metrics. For
$\hat g=e^{2\Lambda}dr^2+r^2h_Y$ and $N=N(r)$,
\[
\Ric(e_0,e_0)=N^{-1}\hat\Delta N,\qquad \Ric|_{\rm spatial}=\widehat{\Ric}-N^{-1}\hat\nabla^2N,
\]
where $\hat\Delta=-\tr_{\hat g}\hat\nabla^2$ and
\[
\hat\Delta N=-e^{-\Lambda}r^{-d}\big(r^de^{-\Lambda}N'\big)',\quad
\hat\nabla^2N(\partial_r,\partial_r)=N''-\Lambda'N',\quad
\hat\nabla^2N|_Y=r\,e^{-2\Lambda}N'\,h_Y,
\]
\[
\widehat{\Ric}(\partial_r,\partial_r)=\frac{d\,\Lambda'}r,\qquad
\widehat{\Ric}|_Y=\Big[(d-1)-e^{-2\Lambda}\big((d-1)-r\Lambda'\big)\Big]h_Y .
\]

\subsection{The linearized equations}
Set $\bar h=2\alpha P_0+2\beta P_1$ and let $\Phi_0,\Phi_1,\Phi_2$ be the coefficients of
$2\,D\Ric(g)\bar h$ on $P_0,P_1,P_2$. Linearizing the formulas above at $N=\sqrt f$,
$e^{-2\Lambda}=f$, and using $(d-1)f+rf'=d-1$, one gets
\begin{align}
\Phi_0&=-\frac2{r^d}\Big(r^df\alpha'+\frac p2(\alpha-\beta)\Big)',\label{eq:Phi0}\\
\Phi_1&=\frac{2df}r\beta'-2f\alpha''-3f'\alpha'+f'\beta',\label{eq:Phi1}\\
\Phi_2&=\frac{4(d-1)}{r^2}\beta+\frac{2f}r(\beta'-\alpha').\label{eq:Phi2}
\end{align}
For an invariant source $S=S_0P_0+S_1P_1+S_2P_2$ and a radial field $\xi=\xi(r)\partial_r$,
\[
\mathcal L_\xi g=\frac{\xi f'}fP_0+\Big(2\xi'-\frac{\xi f'}f\Big)P_1+\frac{2\xi}rP_2,
\]
and after elimination of $\varphi$ the equation $2\,D\Ric(g)\bar h=S+\varphi g+\mathcal L_\xi g$
becomes
\begin{align}
\Phi_0-\Phi_1&=S_{01}-2f\Big(\frac\xi f\Big)',\label{eq:01}\\
\Phi_1-\Phi_2&=S_{12}+2\xi'-\xi\Big(\frac{f'}f+\frac2r\Big),\label{eq:12}
\end{align}
with $S_{ij}=S_i-S_j$. The key point is that second derivatives cancel in \eqref{eq:01}:
\begin{equation}\label{eq:constraint}
\boxed{\;\Phi_0-\Phi_1=-\frac{2df}r\,(\alpha+\beta)'\;}
\end{equation}
This is the constraint. Finally, with $\sigma=\alpha+\beta$ and $\beta=\sigma-\alpha$,
\begin{equation}\label{eq:Lalpha}
\Phi_1-\Phi_2=\mathcal L[\alpha]+\Big(\frac{2(d-1)f}r+f'\Big)\sigma'-\frac{4(d-1)}{r^2}\sigma,
\end{equation}
where
\[
\mathcal L[\alpha]=-2f\alpha''-\Big(\frac{2(d-2)f}r+4f'\Big)\alpha'+\frac{4(d-1)}{r^2}\alpha .
\]

\begin{lemma}[Homogeneous solutions]\label{lem:homog}
The equation $r^2f\alpha''+\big((d-2)rf+2r^2f'\big)\alpha'-2(d-1)\alpha=0$, i.e.
$\mathcal L[\alpha]=0$, has the solutions
\[
\alpha_1=\frac1{r^p-1}\ \ (\text{variation of the mass}),\qquad
\alpha_2=\frac{r^{p+2}-1}{r^p-1}\ \ (\text{growing mode}),
\]
with Wronskian $\alpha_1\alpha_2'-\alpha_1'\alpha_2=(p+2)r^{p+1}/D^2$. Moreover $\alpha=\sigma=c$
constant (rescaling of $\tau$) solves the full system without source. At infinity the
characteristic exponents are $2$ and $-p$.
\end{lemma}

\subsection{Smoothness at the bolt}
Near $r=1$, we have $r-1\sim\varrho^2$, where $\varrho$ is the distance to the bolt. A tensor
$a\,e_\tau^2+b\,e_r^2+c\,P_2$ is smooth if and only if $a,b,c$ are smooth functions of $r$ and
$a(1)=b(1)$; in that case, $a-b=(r-1)\cdot(\text{smooth})$. A field $\zeta(r)\partial_r$ is
smooth if and only if $\zeta$ is smooth and $\zeta(1)=0$.

%====================================================================
\section{First step: $\Delta_Lh=k+\Lc X$}\label{sec:h}
%====================================================================

We take $S=-fP_0$, so that $S_{01}=-f$ and $S_{12}=0$, and we choose the gauge
$\xi=-\frac12rf\,\partial_r=\nabla w$, $w=-r^2/4$, in order to have $-2f(\xi/f)'=f$. Then
\eqref{eq:01} and \eqref{eq:constraint} give
\[
(\alpha+\beta)'=0,\qquad\text{i.e.}\qquad\beta=c_0-\alpha .
\]
The right-hand side of \eqref{eq:12} equals $2\xi'-\xi(f'/f+2/r)=-\frac12rf'=-\frac p2r^{-p}$,
and \eqref{eq:12}--\eqref{eq:Lalpha}, multiplied by $-r^2/2$, become the scalar ODE
\begin{equation}\label{eq:ODE1}
\boxed{\;r^2f\,\alpha''+\big((d-2)rf+2r^2f'\big)\alpha'-2(d-1)(\alpha-c_0)=\frac p4\,r^{2-p}\;}
\end{equation}
For $n=4$ ($p=1$, $d=2$), this reads $r(r-1)\alpha''+2\alpha'-2\alpha=\frac r4-2c_0$.

We then look for $\alpha=\sum_{j\ge0}B_jr^{2-(j+1)p}$. With the indicial polynomial
$C(s)=(s+p)(s-2)$ and the coupling coefficient $D(s)=s(p+2-s)$, we get $C(s_0)B_0=p/4$ with
$s_j=2-(j+1)p$, and then $C(s_j)B_j+D(s_{j-1})B_{j-1}=0$, whence $B_j=B_{j-1}=-\frac18$ for all
$j$ (at the resonances $jp=2$, the coefficient is free and corresponds to $\alpha_1$). The series
sums to
\[
\alpha_{\rm part}=-\frac18\,\frac{r^2}{r^p-1}.
\]
Adding $\frac18\alpha_1$ to remove the pole at $r=1$, the general solution is
\[
\alpha=c_0-A(r)+c_1\alpha_1+c_2\alpha_2,\qquad
\boxed{A(r)=\frac{r^2-1}{8(r^p-1)}}\qquad\big(A(1)=\tfrac1{4p}\big).
\]

For smoothness we take $c_1=0$ (to avoid a pole at $r=1$), and $c_2=0$ to avoid growth like
$r^2$. The condition $2\alpha(1)=2\beta(1)$ gives $c_0=\frac1{2p}$, hence
\[
\alpha=\frac1{2p}-A,\qquad\beta=A,\qquad\bar h=\Big(\frac1p-2A\Big)P_0+2A\,P_1 .
\]

\nv{For $n\ge6$, $A\to0$ at infinity and $[\bar h]_0$ tends to the constant tensor
$\frac1p[P_0]_0$; for $n=5$, $A\equiv\frac18$ and $[\bar h]_0=\frac14[P_0+P_1]_0$. In both cases
the term $\frac1p[P_0]_0$ is removed} by adding the gauge term $\Lc Z$ with
\[
Z=\frac1{4p}\nabla(r^2)=\frac{rf}{2p}\,\partial_r\qquad(\text{smooth, }Z|_{r=1}=0),
\]
which is allowed since $\Delta_L\Lc Z=\Lc\Delta_HZ$.

We thus obtain the following solution for $n\ge5$.
\begin{theorem}\label{thm:h}
For $n\ge5$ and $r_0=1$, the tensor
\[
h=\Big[\,2A\,(P_1-P_0)+r^{-p}\Big(\big(\tfrac1p+\tfrac12\big)P_0+\tfrac12P_1\Big)\Big]_0
\]
is smooth, trace-free, and satisfies $\Delta_Lh=k+\Lc X$ with $X=\chi(r)\partial_r$,
\[
\chi=\frac{3(n-2)}{8(n-1)}\big(r^{1-p}-r^{-1-p}\big)-2fA'-\frac{2(p+1)}rfA-2p\,r^{-1-p}A
+r^{-1-p}\big(r^{-p}-\tfrac12\big).
\]
In frame components, with $\mu=\frac{(n-2)\,r^{-p}}{n(n-3)}$:
$h_{\tau\tau}=-2A+(\frac1p+\frac12)r^{-p}-\mu$, $h_{rr}=2A+\frac12r^{-p}-\mu$, $h|_Y=-\mu$.
Moreover $h=-2A(P_0-P_1)+O(r^{3-n})=O(r^{5-n})$; for $n=5$, $A\equiv\frac18$.
\end{theorem}

For $n=4$, i.e.\ $p=1$, $A=\frac{r+1}8$ grows linearly, and we correct this with a gauge field
$\zeta(r)\partial_r$, $\zeta=-\frac{r^2}8+Br+C$, $\zeta(1)=0$. A direct computation yields a
one-parameter family of smooth solutions, and the choice $B=\frac18$ yields a tensor with constant
coefficients:
\begin{theorem}\label{thm:h4}
For $n=4$, with $r_0=2m$, the tensor
\[
h=\frac{3m^2}4\big(g-2r^2g_{S^2}\big)
\]
satisfies $\Delta_Lh=k+\Lc X$, with $X=-\frac{2m^2f}r\,\partial_r=-2m^2\nabla\ln r$.
\end{theorem}

%====================================================================
\section{Second step: $\Delta_LH=h+\Lc\tilde X$}\label{sec:H}
%====================================================================

\subsection{The general case}
We apply the same method with $S=h+\mu g$, the representative with $S_2=0$:
\[
S_{01}=-4A+\frac{r^{-p}}p,\qquad S_{02}=-2A+\Big(\frac1p+\frac12\Big)r^{-p}.
\]
This time, we take $\xi=0$. The constraint \eqref{eq:constraint} gives $\sigma$ directly:
\begin{equation}\label{eq:ODEsigma}
\boxed{\;\sigma'=-\frac{r\,S_{01}}{2d\,f}\;}
\end{equation}
and \eqref{eq:12}--\eqref{eq:Lalpha} together with $S_{12}=S_{02}-S_{01}$ give the ODE
\begin{equation}\label{eq:ODE2}
\boxed{\;\mathcal L[\alpha]=R(r):=S_{02}+\Big(\frac{2f}r-f'\Big)\sigma'+\frac{4(d-1)}{r^2}\sigma\;}
\end{equation}
that is, $r^2f\alpha''+\big((d-2)rf+2r^2f'\big)\alpha'-2(d-1)\alpha=-\frac{r^2}2R$.

We first solve the constraint. Let
\[
J(r)=\int_1^r\frac{s(s^2-1)}{s^p-1}\,ds,\qquad M(r)=\int_1^r\frac{s^2-1}{s(s^p-1)}\,ds .
\]
Using \nv{$\frac{d}{ds}\big(s^a/D(s)\big)=a\,s^{a-1}/D-p\,s^{a+p-1}/D^2$}, integration of
\eqref{eq:ODEsigma} gives
\[
\sigma=-\frac{r^2(r^2-1)}{4dp\,D}+\frac J{dp}+\sigma_0 .
\]
For $p\ge4$, summing over the $p$-th roots of unity $\omega\ne1$,
\[
J=\frac1p\sum_\omega(\omega^4-\omega^2)\ln\frac{r-\omega}{1-\omega},\qquad
M=\ln r+\frac1p\sum_\omega(\omega^2-1)\ln\frac{r-\omega}{1-\omega},
\]
\nv{where $\ln$ is the principal determination: since $\mathrm{Re}\,(1-\omega)>0$, the quantity
$\frac{r-\omega}{1-\omega}=1+\frac{r-1}{1-\omega}$ has positive real part for $r\ge1$. The terms
associated with $\omega$ and $\bar\omega$ are complex conjugate, so that the sums are real.}
Special cases: $p=2$: $J=\frac{r^2-1}2$ and $M=\ln r$; $p=3$:
$J=r-1-\frac2{\sqrt3}\Big(\arctan\frac{2r+1}{\sqrt3}-\frac\pi3\Big)$; $p=4$:
$J=\frac12\ln\frac{r^2+1}2$.

We now solve the ODE for $\alpha$. With $c_d=\frac1{2d}$ one gets, after simplification,
\[
R=c_d\Big[(p+2)r^{-p}-3\frac{J'}r+\frac12\Big(\frac{r^3-r}D\Big)'+\frac{8J}{r^2}+\frac4{p\,r^2}\Big]
\qquad(\text{for }\sigma_0=\tfrac1{2dp^2}).
\]
Variation of constants (Lemma~\ref{lem:homog}) gives
\[
\alpha=-\frac1{2(p+2)D}\int_1^r\big(r^{p+2}-s^{p+2}\big)R(s)\,\frac{ds}s\ +\ \text{homogeneous}.
\]
The integrals are computed by parts; all terms in $J/s^3$ and $s^{p-1}J$ reduce to $J$, $M$ and
rational functions. One finds
\[
\alpha=-\frac {c_d}{2(p+2)}\,\frac{\mathcal N(r)-\mathcal N(1)}D-\frac{c_d(p+3)}{2p(p+2)}\alpha_2
+c_2\alpha_2+\tilde c,
\]
\begin{multline*}
\mathcal N(r)=\frac{p^2+7p+16}{8p}r^4-\frac{3(p^2+5p+8)}{4p}r^2-\frac{2(p+2)}{p^2}r^p\\
+\Big(\frac{p^2+7p+16}{2p}-\frac{4(p+2)}pr^p\Big)J+\frac32r^{p+2}M,
\end{multline*}
\[
\mathcal N(1)=-\frac{5p^3+23p^2+48p+32}{8p^2},\qquad
\tilde c=-\frac{2(p+2)}pc_2 ,
\]
\[
\sigma=-\frac{r^2(r^2-1)}{4dpD}+\frac J{dp}+\frac1{2dp^2}+\tilde c .
\]

Note that smoothness at the bolt is automatic: $\alpha(1)=0$ by construction and $\sigma(1)=0$;
the shift $\tilde c$ (rescaling of $\tau$) compensates for the addition of $c_2\alpha_2$.

As for the growth, we choose:
\begin{itemize}
\item $n=5$: $c_2=0$;
\item $n\ge6$: the choice $c_2=\frac{c_d(p+3)}{2p(p+2)}+\frac{3c_d\,M_\infty}{4(p+2)}$ kills the
$r^2$ term, with $M_\infty=\frac1p\big(\psi(1)-\psi(1-\frac2p)\big)$, $\psi$ the digamma function;
\item $n\ge8$: the remaining constant is removed by the gauge $Z=a_\infty\,rf\,\partial_r$, with
$a_\infty=\frac{J_\infty}{dp}+\frac1{2dp^2}+\tilde c$ and
$J_\infty=\frac1p\big(\psi(1-\frac2p)-\psi(1-\frac4p)\big)$.
\end{itemize}

At this stage we have proved:
\begin{theorem}\label{thm:H}
For $n\ge5$, the tensor
\[
H=2\alpha P_0+2(\sigma-\alpha)P_1-\frac{2\sigma}ng+\Lc Z
\]
is smooth, trace-free, and satisfies $\Delta_LH=h+\Lc\tilde X$ with
$\tilde X=-W-2p\,a_\infty r^{-p-1}f\,\partial_r$, where $W=-\frac12B(\bar H)$ for
$\bar H=2\alpha P_0+2(\sigma-\alpha)P_1$, that is
\[
W=f\Big[2r^{-d}\big(r^d(\sigma-\alpha)\big)'-\frac{f'}f(2\alpha-\sigma)-\sigma'\Big]\partial_r .
\]
Its growth is $O(r^2\ln r)$ for $n=5$, $O(r)$ for $n=6$, $O(\ln r)$ for $n=7$ and $O(r^{7-n})$
for $n\ge8$.
\end{theorem}

\begin{remark}
For $n=4$ ($p=1$), the formulas of Theorem~\ref{thm:H} still make sense, with
$J=\frac{r^3-1}3+\frac{r^2-1}2$, $M=r-1+\ln r$, $c_2=0$ and $Z=0$, and give a smooth $H$ with
$\Delta_LH=h+\Lc\tilde X$, where $h$ is the tensor of Theorem~\ref{thm:h} with $p=1$. However
this $h$ grows like $r$ (since $A=\frac{r+1}8$) and $H$ grows like $r^3$. This would still be
sufficient for Theorem~\ref{thm:main} (since $3<\frac72$), but we prefer a tensor $H$ growing
only like $r^2\ln r$.
\end{remark}

\subsection{Closed forms in dimensions $4$ and $5$}
For $n=4$ we therefore start from the tensor $h$ of Theorem~\ref{thm:h4}, and we use the gauge
$H|_Y=0$, which also works for $n=5$. This gauge, i.e.\ $\zeta=-r^2\int_1^r\sigma/t^2$, gives a
single function: $H=\phi_H(r)\,(P_0-P_1)$ with $\phi_H=2\alpha+\big(\frac{f'}f-\frac2r\big)\zeta$.

In these two dimensions, we use another representative of the source modulo
$\operatorname{Im}\mathcal L+C^\infty g$, for which $S_{01}=0$; the constraint then gives
$\sigma=c_0$ constant, and \eqref{eq:ODE2} reduces to $\mathcal L[\alpha]=S_{02}+4(d-1)c_0/r^2$.

For $n=4$, the source $h=\frac3{16}(g-2P_2)\equiv-\frac38P_2$ gives $S_{01}=0$,
$S_{02}=\frac38$, and the ODE
\[
r(r-1)\alpha''+2\alpha'-2\alpha=-\frac3{16}r^2-2c_0,
\]
which is resonant with $\alpha_2=r^2+r+1$; this produces a logarithm:
$\alpha=-\frac{r^3\ln r}{16(r-1)}+c_0$ with $c_0=\frac18$ by smoothness. We obtain:
\begin{theorem}\label{thm:H4}
In dimension $4$ (with $r_0=1$), the tensor
\[
H=\frac18\Big(2r-1-\frac{r^3\ln r}{r-1}\Big)\,(P_0-P_1)=O(r^2\ln r)
\]
is smooth and solves $\Delta_LH=h+\Lc\tilde X$, where $h$ is the tensor of
Theorem~\ref{thm:h4} and $\tilde X=\tilde\chi(r)\,\partial_r$ with
\[
\tilde\chi=-\frac12\,r\ln r-\frac r8+\frac12-\frac1{2r}+\frac1{8r^2}.
\]
Moreover $\tilde\chi(1)=0$, so that $\tilde X$ is smooth at the bolt.
\end{theorem}

\begin{proof}[Comments]
$H$ is smooth at the bolt since $\frac{\ln r}{r-1}$ is analytic at $r=1$ and the coefficient of
$P_0-P_1$ vanishes at $r=1$ (it equals $-\frac{r-1}{16}+O((r-1)^2)$).
\end{proof}

Although it is not needed, let us also look at the tensor $H$ obtained in this way when $n=5$.
Since $A\equiv\frac18$, we have $h\equiv\bar h=\frac14(P_0+P_1)$ (Theorem~\ref{thm:h} and the
gauge $Z$), so $S_{01}=0$, $S_{02}=\frac14$, and the ODE
\[
(r^2-1)\alpha''+\Big(r+\frac3r\Big)\alpha'-4\alpha=-\frac{r^2}8-4c_0,
\]
with homogeneous solutions $\frac1{r^2-1}$ and $r^2+1$, gives
$\alpha=-\frac{r^4\ln r}{32(r^2-1)}+c_0$, $c_0=\frac1{32}$. Then the tensor
\[
H=\frac1{16}\Big(r-\frac1{r+1}-\frac{r^4\ln r}{r^2-1}\Big)\,(P_0-P_1)=O(r^2\ln r)
\]
is smooth (its coefficient equals $-\frac{r-1}{64}+O((r-1)^2)$) and solves
$\Delta_LH=h+\Lc\tilde X$ with $\tilde X=\tilde\chi(r)\,\partial_r$,
\[
\tilde\chi=-\frac5{16}\,r\ln r-\frac5{16}\,r+\frac5{32}+\frac1{4r}-\frac5{32\,r^2}+\frac1{16\,r^3},
\]
and $\tilde\chi(1)=0$, so that $\tilde X$ is smooth at the bolt.

\begin{remark}
The gauge $H|_Y=0$ is compatible with the allowed growth only when $n\le6$.
\end{remark}

%====================================================================
\section{Consequence: $P(H)=k$}\label{sec:PH}
%====================================================================
For $\Ric(g)=0$, the operator of the TT-tensor machine \cite{DelTT} is
\[
P=\Delta_L^2-\Lc\,Q\,\dv,\qquad Q=d^*d+\frac n{2(n-1)}dd^* .
\]
It maps trace-free symmetric tensors to TT tensors, and $P\Lc=0$.

\begin{theorem}\label{thm:PH}
Let $H$ be the tensor of Theorem~\ref{thm:H} for $n\ge5$, or of Theorem~\ref{thm:H4} for $n=4$.
Then $P(H)=k$.
\end{theorem}

\begin{proof}
By Lemma~\ref{lem:comm}, $\Delta_L^2H=k+\Lc(X+\Delta_H\tilde X)$, hence $P(H)=k+\Lc W_0$ with
$W_0=X+\Delta_H\tilde X-Q\dv H$. Since $P(H)$ and $k$ are TT, $\dv\Lc W_0=0$. Applying $d^*\dv$ to
both equations and using $\dv\Delta_L=\Delta_H\dv$ and $d^*\dv\Lc=\frac{2(n-1)}n\Delta d^*$, one
gets $d^*W_0=-\vartheta$, where $\vartheta=\frac n{2(n-1)}d^*\dv h-d^*X$ is harmonic.
\nv{Since $\vartheta$ is radial, $r^df\vartheta'$ is constant; this constant vanishes since
$\vartheta$ is smooth at the bolt, where $f=0$. Hence $\vartheta$ is constant. Moreover
$d^*\dv h$ and $d^*X$ tend to $0$ at infinity: $h$ has two derivatives decaying at least like
$r^{-2}$ (Corollary~\ref{cor:types}), and $d^*X=-r^{-d}(r^d\chi)'=O(r^{-p})$ since
$\chi=O(r^{1-p})$. Thus $\vartheta\equiv0$.}

\nv{All the tensors and fields involved ($k$, $h$, $H$, $X$, $\tilde X$) are built from $r$,
$P_0$, $P_1$, $P_2$ and radial fields, hence are invariant under $U(1)\times O(n-1)$ and under the
reflection $\tau\mapsto-\tau$, and so is $W_0$. An invariant $1$-form is of the form
$a(r)\,dr+b(r)\,d\tau$, and the reflection forces $b=0$ (note that $\kappa=f\,d\tau$ is invariant
under $U(1)\times O(n-1)$ but odd under the reflection).} Thus the $1$-form $W_0$ is
radial, hence closed, and it is co-closed; thus $W_0=c\,(r^df)^{-1}dr$, which is smooth at the bolt
only if $c=0$. Hence $W_0=0$.

The same proof applies to the closed forms of Section~\ref{sec:H} for $n=4,5$, since $h$, $H$,
$X$ and $\tilde X$ are smooth at the bolt. For $n=4$ \nv{(with $r_0=1$, i.e.\ $m=\frac12$, so that
$h=\frac3{16}(g-2P_2)$ and $X=-\frac f{2r}\partial_r$)} one even checks directly that
$\vartheta\equiv0$: $\frac n{2(n-1)}d^*\dv h=d^*X=\frac1{2r^2}$.
\end{proof}

%====================================================================
\section{Weighted Sobolev spaces}\label{sec:sobolev}
%====================================================================

We use weighted Sobolev spaces in the spirit of Bartnik \cite{Bar86}. Throughout this section,
$T$ denotes any tensor bundle over $M$ and $N=n-1$. Recall that $r\ge1$ on $M$ and that $r$
is a smooth function on $M$, including at the bolt (since $r-1$ is a smooth function of
$\varrho^2$); we use it as a weight. Since
$d\mu_g=r^{n-2}\,d\tau\,dr\,d\mu_{h_Y}$, the volume of $\{r\le R\}$ grows like $R^N$; $N$ plays the
role of the dimension at infinity.

\begin{definition}
For $q\in\mathbb N$ and $\delta\in\mathbb R$, let $\Hs^q_\delta=\Hs^q_\delta(M,T)$ be the space of
sections $u\in H^q_{\rm loc}$ such that
\[
\|u\|_{q,\delta}^2:=\sum_{j=0}^q\int_M|\nabla^ju|^2\,r^{2(j-\delta)-N}\,d\mu_g<\infty .
\]
\end{definition}

This is a Hilbert space. Roughly speaking, $u\in \Hs^q_\delta$ means that
$\nabla^ju=o(r^{\delta-j})$ in an $L^2$ sense. In particular $\Hs^0_{-N/2}=L^2$, and the energy
space
\[
\mathcal W:=\Hs^1_{\delta_0},\qquad \delta_0=\frac{3-n}2,\qquad
\|u\|_{\mathcal W}^2=\int_M\big(|\nabla u|^2+r^{-2}|u|^2\big)\,d\mu_g,
\]
is the natural domain of $\mathcal Q$.

\subsection{Tensors with symbol-type behavior}

\begin{definition}
Let $a\in\mathbb R$ and $m\in\mathbb N$. A smooth tensor field $u$ on $M$ is of
\emph{type $(a,m)$} if for every $j\ge0$ there is $C_j$ with
\[
|\nabla^ju|\le C_j\,r^{a-j}(1+\ln r)^m\qquad\text{on }M .
\]
\end{definition}

\begin{lemma}\label{lem:type}
\begin{enumerate}
\item The tensors $g$, $e_\tau\otimes e_\tau$, $e_r\otimes e_r$ and $P_2$ are of type $(0,0)$, and
$r$ is of type $(1,0)$.
\item If $\phi(r)$ is a smooth function on $[1,\infty)$ with
$|\phi^{(i)}|\le C_i\,r^{a-i}(1+\ln r)^m$ for all $i$, then $\phi$ (seen as a function on $M$) is
of type $(a,m)$.
\item Products of tensors of types $(a,m)$ and $(b,l)$ are of type $(a+b,m+l)$.
\item The curvature tensor is of type $(-p-2,0)$ \nv{(recall that $h_Y$ is the round metric)}.
\end{enumerate}
\end{lemma}

\begin{proof}
In the frame $\{e_\tau,e_r,r^{-1}\epsilon_i\}$, with $\epsilon_i$ an orthonormal frame of $h_Y$,
the connection coefficients are smooth functions of $r$ and of the point of $Y$, of the form
$O(r^{-1})$ together with all their frame derivatives multiplied by the corresponding powers of
$r^{-1}$: indeed $\frac{f'}{\sqrt f}=O(r^{-p-1})$, $\frac{\sqrt f}r=O(r^{-1})$, and the
connection of $r^2h_Y$ scales like $r^{-1}$. Point 1 follows by induction on $j$, and points 2
and 3 follow from the Leibniz rule and $|\nabla^i r|\le C_ir^{1-i}$. Point 4 holds because the
frame components of the curvature are $O(r^{-p-2})$ functions of $r$ of symbol type.
\end{proof}

\begin{corollary}\label{cor:types}
The tensors $k$, $h$ and $H$ are of type $(a,m)$ with:

\begin{center}
\begin{tabular}{c|c|c|c}
 & $k$ & $h$ & $H$\\\hline
$n=4$ & $(-2,0)$ & $(0,0)$ & $(2,1)$\\
$n=5$ & $(-2,0)$ & $(0,0)$ & $(2,1)$\\
$n=6$ & $(-3,0)$ & $(-1,0)$ & $(1,1)$\\
$n=7$ & $(-4,0)$ & $(-2,1)$ & $(0,1)$\\
$n\ge8$ & $(3-n,0)$ & $(5-n,1)$ & $(7-n,1)$
\end{tabular}
\end{center}
\end{corollary}

\begin{proof}
Their frame components are built out of rational functions of $r$, $r^{-p}$, $\ln r$, $J$ and
$M$, all of which satisfy symbol estimates: the rational functions have expansions in powers of
$r^{-1}$ at infinity, and $J'$, $M'$ are rational. The orders follow from
Theorems~\ref{thm:h}, \ref{thm:h4}, \ref{thm:H} and~\ref{thm:H4}; for simplicity, a logarithmic
factor is allowed in the whole column of $h$\nv{, and for $H$ when $n\ge8$ (Theorem~\ref{thm:H}
gives $O(r^{7-n})$ without logarithm)}.
\end{proof}

\begin{lemma}[Membership]\label{lem:member}
If $u$ is smooth of type $(a,m)$, then $u\in \Hs^q_\delta$ for every $q$ and every $\delta>a$.
\end{lemma}

\begin{proof}
On the compact set $\{r\le2\}$, there is nothing to prove. On $\{r\ge2\}$,
\[
\begin{aligned}
\int_{\{r\ge2\}}|\nabla^ju|^2r^{2(j-\delta)-N}d\mu
&\le C\int_2^\infty r^{2(a-j)}(\ln r)^{2m}\,r^{2(j-\delta)-(n-1)}\,r^{n-2}\,dr\\
&=C\int_2^\infty r^{2(a-\delta)-1}(\ln r)^{2m}\,dr<\infty .
\end{aligned}
\qedhere
\]
\end{proof}

\subsection{Continuity of the covariant derivative and of differential operators}

\begin{proposition}[Continuity of $\nabla$]\label{prop:nabla}
For all $q\in\mathbb N$ and $\delta\in\mathbb R$, the covariant derivative is continuous
\[
\nabla:\Hs^{q+1}_\delta\longrightarrow \Hs^q_{\delta-1},\qquad
\|\nabla u\|_{q,\delta-1}\le\|u\|_{q+1,\delta}.
\]
\end{proposition}

\begin{proof}
By definition,
\[
\|\nabla u\|_{q,\delta-1}^2=\sum_{j=0}^q\int|\nabla^{j+1}u|^2r^{2(j-\delta+1)-N}
=\sum_{i=1}^{q+1}\int|\nabla^iu|^2r^{2(i-\delta)-N}\le\|u\|_{q+1,\delta}^2 .\qedhere
\]
\end{proof}

\begin{proposition}[Multiplication]\label{prop:mult}
Let $A$ be a smooth tensor of type $(b,0)$ and let $\ast$ denote any contraction. Then
$u\mapsto A\ast u$ is continuous $\Hs^q_\delta\to \Hs^q_{\delta+b}$.
\end{proposition}

\begin{proof}
By the Leibniz rule, $|\nabla^j(A\ast u)|\le C\sum_{i\le j}r^{b-(j-i)}|\nabla^iu|$, hence
\[
\int|\nabla^j(A\ast u)|^2r^{2(j-\delta-b)-N}\le C\sum_{i\le j}\int|\nabla^iu|^2r^{2(i-\delta)-N}.
\qedhere
\]
\end{proof}

\begin{corollary}[Differential operators]\label{cor:diffop}
Let $L=\sum_{j=0}^mA_j\ast\nabla^j$ with $A_j$ smooth of type $(j-m,0)$. Then
$L:\Hs^{q+m}_\delta\to \Hs^q_{\delta-m}$ is continuous. In particular
\[
P:\Hs^{q+4}_\delta\longrightarrow \Hs^q_{\delta-4}
\]
is continuous for all $q,\delta$.
\end{corollary}

\begin{proof}
The first assertion follows from Propositions~\ref{prop:nabla} and~\ref{prop:mult}, together
with the continuous inclusion $\Hs^q_{\delta}\subset \Hs^q_{\delta_1}$ for $\delta\le\delta_1$
(since $r\ge1$). For $P=\Delta_L^2-\Lc Q\dv$, expanding gives $P=\sum_{j\le4}A_j\ast\nabla^j$
where $A_4$ is built from $g$ (type $(0,0)$), \nv{$A_3=0$ (the operators $\nabla^*\nabla$,
$\Lc$, $Q$ and $\dv$ are contractions of iterated covariant derivatives, and curvature only appears
through the terms $\Rm$ of $\Delta_L$), and $A_j$, $j\le2$, involves $\nabla^{2-j}\Rm$, or
$\Rm\ast\Rm$ for $j=0$}, which by Lemma~\ref{lem:type} are of type $(-p-4+j,0)$ or better, hence
of type $(j-4,0)$ since $p\ge1$.
\end{proof}

\subsection{Cut-off functions}
Fix $\eta\in C^\infty(\mathbb R)$ with $\eta=1$ on $(-\infty,1]$ and $\eta=0$ on $[2,\infty)$, and
set $\eta_R=\eta(r/R)$ for $R\ge1$.

\begin{lemma}\label{lem:cutoff}
For every $i\ge1$ there is $C_i$, independent of $R$, such that
$|\nabla^i\eta_R|\le C_i\,r^{-i}$, and $\supp\nabla\eta_R\subset A_R:=\{R\le r\le2R\}$.
\end{lemma}

\begin{proof}
By the Faà di Bruno formula, $\nabla^i\eta_R$ is a sum of terms
$R^{-l}\eta^{(l)}(r/R)\,\nabla^{i_1}r\otimes\dots\otimes\nabla^{i_l}r$ with
$i_1+\dots+i_l=i$. On $A_R$, $R^{-l}\le2^lr^{-l}$, and $|\nabla^{i_s}r|\le Cr^{1-i_s}$
by Lemma~\ref{lem:type}.
\end{proof}

\begin{proposition}[Convergence of cut-offs]\label{prop:cutoff}
For every $u\in \Hs^q_\delta$, $\eta_Ru\to u$ in $\Hs^q_\delta$ as $R\to\infty$. In particular
compactly supported smooth tensors are dense in $\Hs^q_\delta$.
\end{proposition}

\begin{proof}
Write $v_R=(1-\eta_R)u$. By the Leibniz rule and Lemma~\ref{lem:cutoff},
\[
|\nabla^jv_R|\le(1-\eta_R)|\nabla^ju|+C\sum_{i=1}^jr^{-i}|\nabla^{j-i}u|\,\mathbf 1_{A_R}.
\]
The first term tends to $0$ in $L^2(r^{2(j-\delta)-N}d\mu)$ by dominated convergence. For the
second,
\[
\begin{aligned}
\int_{A_R}r^{-2i}|\nabla^{j-i}u|^2r^{2(j-\delta)-N}
&=\int_{A_R}|\nabla^{j-i}u|^2r^{2(j-i-\delta)-N}\\
&\le\int_{\{r\ge R\}}|\nabla^{j-i}u|^2r^{2(j-i-\delta)-N}\longrightarrow0,
\end{aligned}
\]
since it is the tail of a convergent integral. Summing over $j\le q$ gives $\|v_R\|_{q,\delta}\to0$.
The density statement follows by combining this with standard mollification on the compact set
$\{r\le2R\}$.
\end{proof}

\subsection{Continuity of the quadratic form}

\begin{proposition}\label{prop:Qcont}
$\mathcal Q$ extends to a continuous quadratic form on $\mathcal W=\Hs^1_{\delta_0}$:
\[
|\mathcal Q(u)-\mathcal Q(v)|\le C\big(\|u\|_{\mathcal W}+\|v\|_{\mathcal W}\big)\|u-v\|_{\mathcal W}.
\]
\end{proposition}

\begin{proof}
By Lemma~\ref{lem:type}, $|\Rm|\le Cr^{-p-2}\le Cr^{-2}$, hence
$|\langle\Rm u,v\rangle|\le Cr^{-2}|u||v|$, and the bilinear form
$(u,v)\mapsto\int\langle\nabla u,\nabla v\rangle-2\langle\Rm u,v\rangle$ is bounded on
$\mathcal W$ by Cauchy--Schwarz.
\end{proof}

%====================================================================
\section{The unstable direction}\label{sec:main}
%====================================================================
We are now in position to construct our compactly supported destabilizing direction.
\begin{theorem}\label{thm:main}
Let $n\ge4$, let $g$ be the Riemannian Schwarzschild--Tangherlini metric and let $H$ be the tensor
of Theorem~\ref{thm:H} for $n\ge5$, and of Theorem~\ref{thm:H4} for $n=4$. The tensor
\[
k_R:=P(\eta_RH)
\]
is smooth, compactly supported in $\{r\le2R\}$, transverse-traceless,  and converge to  $k$, so for $R$ large enough, it satisfies
$$\mathcal Q(k_R)<0.$$
\nv{More precisely, with $\delta=\frac{11-n}2$ and $(a_H,m)$ the type of $H$ given by
Corollary~\ref{cor:types}, there is $C>0$ such that, for $R\ge2$,
\[
\|k_R-k\|_{\mathcal W}\le C\,R^{a_H-\delta}(\ln R)^m,\qquad
|\mathcal Q(k_R)-\mathcal Q(k)|\le C\,R^{a_H-\delta}(\ln R)^m ,
\]
where $m=1$ and
\[
a_H-\delta=-\tfrac32\ (n=4),\quad -1\ (n=5),\quad -\tfrac32\ (n=6),\quad -2\ (n=7),\quad
\tfrac{3-n}2\ (n\ge8).
\]}
\end{theorem}

\begin{proof}
 The tensor $\eta_RH$ is smooth, trace-free and supported in
$\{r\le2R\}$. Since $P$ is a differential operator that maps trace-free tensors to TT
tensors \cite{DelTT}, $k_R$ is smooth, TT and supported in $\{r\le2R\}$. It is explicit since
$H$ is.
 By Theorem~\ref{thm:PH} and linearity,
\[
k_R-k=P(\eta_RH)-P(H)=-P\big((1-\eta_R)H\big).
\]
Let $\delta=\delta_0+4=\frac{11-n}2$. By Corollary~\ref{cor:types}, $H$ is of type $(a_H,m)$ with
\[
a_H=2\ (n=4,5),\qquad a_H=1\ (n=6),\qquad a_H=0\ (n=7),\qquad a_H=7-n\ (n\ge8),
\]
and in every case $a_H<\frac{11-n}2$ (for $n\ge8$, $7-n<\frac{11-n}2$ is equivalent to $n>3$).
Lemma~\ref{lem:member} therefore gives $H\in \Hs^5_\delta$. By Proposition~\ref{prop:cutoff},
$(1-\eta_R)H\to0$ in $\Hs^5_\delta$, and by Corollary~\ref{cor:diffop} (with one derivative in the
target and $m=4$),
\[
\|k_R-k\|_{\mathcal W}=\big\|P\big((1-\eta_R)H\big)\big\|_{1,\delta_0}\le C\,\|(1-\eta_R)H\|_{5,\delta}
\longrightarrow0 .
\]
In particular $k=P(H)\in\mathcal W$. By Proposition~\ref{prop:Qcont},
$\mathcal Q(k_R)\to\mathcal Q(k)$, and $\mathcal Q(k)<0$ by Theorem~\ref{thm:Ozuch}. Hence
$\mathcal Q(k_R)<0$ for $R$ large.

\nv{For the rate, note that by the Leibniz rule and Lemma~\ref{lem:cutoff},
$|\nabla^j((1-\eta_R)H)|\le C\,r^{a_H-j}(1+\ln r)^m\,\mathbf 1_{\{r\ge R\}}$ for $j\le5$, with $C$
independent of $R$. As in the proof of Lemma~\ref{lem:member}, since $a_H<\delta$,
\[
\|(1-\eta_R)H\|_{5,\delta}^2\le C\int_R^\infty r^{2(a_H-\delta)-1}(1+\ln r)^{2m}\,dr
\le C\,R^{2(a_H-\delta)}(\ln R)^{2m}\qquad(R\ge2).
\]
The estimate of $\|k_R-k\|_{\mathcal W}$ follows from Corollary~\ref{cor:diffop}, and that of
$|\mathcal Q(k_R)-\mathcal Q(k)|$ from Proposition~\ref{prop:Qcont}, since $\|k_R\|_{\mathcal W}$
is bounded. The values of $a_H-\delta$ follow from Corollary~\ref{cor:types}.}
\end{proof}

\begin{remark}
\begin{enumerate}
\item Cutting off $k$ directly would not work, since $\eta_Rk$ is not divergence-free. The point
of the construction is that the cut-off is applied to the potential $H$, and that $P$ restores the
TT condition exactly.
\item For $n=5$ one has $k=O(r^{-2})$ and $k\notin L^2$, but $k\in\mathcal W$; this is why the
energy space $\mathcal W$, rather than $L^2$, is the natural framework.
\item Since $k_R$ is smooth, compactly supported and TT with $\mathcal Q(k_R)<0$, the min-max
principle shows that the bottom of the spectrum of (the Friedrichs extension of) $\Delta_L$ on
$L^2$ TT tensors is negative. This is stronger than the instability in the sense of
\cite{BiqOz26,Del26,Ozu26}, where the destabilizing tensor is not compactly supported \nv{(and
not even in $L^2$ when $n=5$, see point~2). Note that deducing the negativity of the bottom of
the spectrum of the Friedrichs extension from a TT tensor which is not compactly supported, even
if it lies in $L^2$, requires precisely a density statement of the type established here.} For
$n=4$, it gives an explicit, smooth, compactly supported alternative to the numerical $L^2$
eigentensor of \cite{GPY82}.
\end{enumerate}
\end{remark}

%====================================================================
\section{Perspectives}\label{sec:persp}
%====================================================================

\subsection{Existence of the preimages by the Fredholm alternative}\label{subsec:fredholm}

The proof of Theorem~\ref{thm:main} only relies on the existence of a smooth trace-free tensor
$H$ with
\[
P(H)=k\qquad\text{and}\qquad H\in \Hs^5_{\delta},\quad \delta=\tfrac{11-n}2 .
\]
The explicit formulas of Sections~\ref{sec:h} and~\ref{sec:H} provide one way of obtaining such
an $H$. Its existence should also follow from the Fredholm alternative in weighted spaces,
without explicit formulas. We outline the expected argument, which has not yet been carried out in
detail for the two tensors $h$ and $H$\nv{; the whole discussion of this section and of the next
one is therefore conditional}.

\paragraph{Reduction to TT tensors.} On a Ricci-flat manifold, $\Delta_L$ commutes with $\dv$ and
with $\Lc$ (Lemma~\ref{lem:comm}). Writing $h=h^{TT}+\Lc Z$ (York decomposition in weighted
spaces), the equation $\Delta_L h\equiv k$ modulo $\operatorname{Im}\Lc$ reduces, since $k$ is TT,
to
\[
\Delta_L\,h^{TT}=k\qquad\text{on TT tensors},
\]
the gauge part $\Lc Z$ being arbitrary. The same holds for the second step with source $h$, after
projection of $h$ onto its TT part.

\paragraph{Fredholm theory.} The metric $g$ is asymptotic to $\mathbb R^{N}\times S^1$, $N=n-1$,
with curvature decaying like $r^{-p-2}$. Since $g$ is Riemannian, $\Delta_L=\nabla^*\nabla-2\Rm$
is asymptotic to the (nonnegative) flat Laplacian of $\mathbb R^N\times S^1$ acting componentwise,
$\Delta_L\simeq-\partial_\tau^2-\sum_{i=1}^N\partial_i^2$. On the Fourier mode
$e^{2i\pi j\tau/\beta}$, $\beta$ the period of $\tau$, it becomes
$\Delta_{\mathbb R^N}+(2\pi j/\beta)^2$ with $\Delta_{\mathbb R^N}=-\sum\partial_i^2\ge0$: for
$j\ne0$, this is a positive mass term, so these modes are invertible, with exponential decay. On
the $\tau$-independent mode ($j=0$), $\Delta_L$ is asymptotic to the flat Laplacian
$\Delta_{\mathbb R^N}$ acting componentwise. \nv{The end is of ALF type (an $S^1$-bundle over
$\mathbb R^N$), which is neither exactly asymptotically conical in the sense of Lockhart--McOwen
\cite{LM85} nor asymptotically flat in the sense of Bartnik \cite{Bar86} in dimension $n$. An
adapted framework is the fibred boundary calculus of Hausel--Hunsicker--Mazzeo \cite{HHM04}, or
the weighted (Hardy type) inequalities of Minerbe on ALF manifolds \cite{Min09}. In such a
framework, one expects that}
\[
\Delta_L:\Hs^{q+2}_\delta\longrightarrow \Hs^q_{\delta-2}
\]
is Fredholm as soon as $\delta$ avoids the indicial set
$I=\{\ell,\ 2-N-\ell\ :\ \ell\in\mathbb N\}$ of the flat Laplacian (shifted, if necessary, by the
indicial roots coming from the tensorial structure)\nv{, with} range the $L^2$-annihilator of the
kernel of $\Delta_L$ in the dual space $\Hs^0_{2-N-\delta}$.

\paragraph{The two steps.}
\begin{itemize}
\item \emph{First step.} The source $k$ is $O(r^{a_k})$ with $a_k=3-n$ ($a_k=-2$ if $n=4$), so
$k\in \Hs^0_{\delta-2}$ for every $\delta>a_k+2$. We choose $\delta\notin I$ slightly above
$a_k+2$. The obstruction space consists of TT solutions $w$ of $\Delta_Lw=0$ with
$w\in \Hs^0_{2-N-\delta}$, i.e.\ decaying faster than $r^{-N-a_k}$. Such a $w$ would be an
$L^2$ static TT zero mode of $\Delta_L$; its absence is expected (it is a linearized version of
the uniqueness of static black holes), but it has to be checked. If the obstruction space is
trivial, there exists $h\in \Hs^{2}_\delta$ solving the problem. Elliptic regularity gives
smoothness, and averaging over \nv{$U(1)\times O(n-1)$} gives invariance.
\item \emph{Second step.} The same argument with source $h$ and weight $\delta+2$ gives $H$. The
dual weight is now even more negative, so the obstruction space can only be smaller.
\end{itemize}

\paragraph{Logarithms.} When the target weight has to cross an indicial root, the solution
picks up a logarithmic term. This is exactly what happens in the explicit solutions of
Section~\ref{sec:H}: for $n=4$ ($N=3$) and $n=5$ ($N=4$), the growth $r^2$ is an indicial
exponent of the flat Laplacian, and $H$ contains $r^2\ln r$.

What remains to be made explicit is the precise indicial set for $\Delta_L$ acting on symmetric
$2$-tensors on these ends, and the triviality of the obstruction spaces for both steps.

\subsection{Kerr and Myers--Perry metrics}

Ozuch's construction \cite{Ozu26} provides an explicit destabilizing tensor on the Riemannian
Myers--Perry metrics (including Riemannian Kerr in dimension $4$), built from the Killing field
$K=\partial_\tau$. For $n\ge5$ his formula gives, with $V=|K|^2$,
$\mathcal Q(k)=-\frac{3(n+2)}{4(n-1)}\int|dV|^2-\frac{n-5}2\int|F|^2<0$. In dimension $4$, the
instability of Riemannian Kerr follows from his main theorem, since Kerr is not locally
hyperkähler.

These metrics have the same type of end as the Schwarzschild--Tangherlini metrics (asymptotic to
a flat $S^1$-bundle over $\mathbb R^N$, possibly twisted), and their curvature decays at the same
rate. \nv{If the Fredholm argument of Section~\ref{subsec:fredholm} can be completed, it would
apply in the same way, under the same two conditions} (indicial set and triviality of the
obstruction spaces). It \nv{would give} a smooth trace-free $H$ with $P(H)=k$ and the required
growth, and Sections~\ref{sec:sobolev} and~\ref{sec:main}, which only use the decay of the
curvature and the growth of $H$, \nv{would} then yield a smooth, compactly supported TT unstable
direction $P(\eta_RH)$.

An explicit construction is harder, since these metrics have cohomogeneity two.

\subsection{The operator $P$ as a numerical tool for instability}

The main difficulty in testing the stability of an Einstein metric numerically is that the
quadratic form has to be evaluated on TT tensors, and that the condition $\dv u=0$ is a
differential constraint which is hard to impose on a finite-dimensional family of test tensors.
The operator $P$ removes this difficulty: for every trace-free $u$, the tensor $P(u)$ is TT, with
no constraint to solve. On an Einstein manifold with $\Ric=\lambda g$ \cite{DelTT},
\[
P=(\Delta_L-2\lambda)\Big(\Delta_L-\frac n{n-1}\lambda\Big)
-\Lc\Big(d^*d+\frac n{2(n-1)}dd^*-\frac n{n-1}\lambda\Big)\dv .
\]
A numerical stability test can then proceed as follows.
\begin{enumerate}
\item \emph{Choose a family of trace-free test tensors} $u_c$, depending on parameters
$c\in\mathbb R^m$: compactly supported bumps, products of a cut-off with natural tensors
(Hessians of functions, products of Killing forms, curvature tensors), or symmetry-adapted
ansätze reducing the problem to fewer variables. The dependence on $c$ need not be linear:
amplitudes, but also centers, widths or cut-off radii may be taken as parameters.
\item \emph{Compute} $v_c=P(u_c)$, which is TT for every $c$. If $u_c$ is compactly supported, so
is $v_c$, and no boundary terms appear.
\item \emph{Minimize the Rayleigh quotient}
\[
\mathcal R(c)=\frac{\displaystyle\int_M\big(|\nabla v_c|^2-2\langle\Rm v_c,v_c\rangle
-2\lambda|v_c|^2\big)}{\displaystyle\int_M|v_c|^2}
=\frac{\langle(\Delta_L-2\lambda)v_c,v_c\rangle_{L^2}}{\|v_c\|_{L^2}^2}
\]
over $c$, by any optimization method (gradient descent, quasi-Newton, Nelder--Mead, \dots).
There is no eigenvalue problem to solve, and there is no need to reach the minimum: as soon as one
value $\mathcal R(c_*)<0$ is found, the explicit tensor $v_{c_*}=P(u_{c_*})$ is a smooth TT
unstable direction. The minimization is only a way to search for such a $c_*$, and the infimum of
$\mathcal R$ over the family gives an upper bound for the bottom of the spectrum of
$\Delta_L-2\lambda$ on TT tensors.
\end{enumerate}

\begin{remark}\leavevmode
\begin{itemize}
\item On a compact Einstein manifold, the image of $P$ is the
$L^2$-orthogonal, in the TT tensors, of $\ker(\Delta_L-2\lambda)|_{TT}+
\ker(\Delta_L-\frac n{n-1}\lambda)|_{TT}$ \cite{DelTT}. Therefore every TT eigentensor with
eigenvalue $<2\lambda$, other than $\frac n{n-1}\lambda$, lies in the image of $P$, and the
procedure detects every such instability, provided the family $(u_c)$ is rich enough. On
noncompact manifolds, Theorem~\ref{thm:main} shows how compactly supported tensors $u=\eta_RH$ can
be used to approximate a given unstable TT direction which is not compactly supported.
\item $P$ has order $4$, so $\mathcal Q(P(u))$ involves five derivatives of $u$. With exact or
automatic differentiation of an explicit ansatz, this is harmless. With finite differences or
finite elements, one should integrate by parts to lower the order, or use the factorization
$P=4\,D\mathrm{Sch}(g)^*\,D\mathring{\Ric}(g)^*$ \cite{DelTT}.
\item The metric itself may be numerical (for instance an Einstein or Ricci-flat metric computed
by a Ricci-flow or Newton method). Only $g$, its curvature and their derivatives up to order $2$
enter $P$, so the procedure applies directly.
\item Since $P(\Lc Z)=0$, the gauge directions of $u$ are invisible, which reduces the effective
dimension of the test family.
\item A single explicit $u$ with $\mathcal Q(P(u))<0$, evaluated with rigorous quadrature or
interval arithmetic, yields a computer-assisted proof of instability.
\end{itemize}
\end{remark}

%====================================================================

\section*{Acknowledgements}
This article was partially funded by the
French National Research Agency ANR-23-CE40-0010: Einstein constraints:
past, present, and future (EINSTEIN-PPF) and ANR-24-CE40-0702: Orbifold, scalar curvature and
rigidity (OrbiScaR). 

\section*{AI usage statement}
The AI assistant Claude (Anthropic) was used to solve the ordinary differential
equations of Sections~\ref{sec:h} and~\ref{sec:H}, and to help write this article in English.
All results were checked by the author.

\medskip
\noindent {\bf Data Availability}: Data sharing is not applicable to this article as no
data sets were generated or analyzed during the current study.

\medskip
\noindent {\bf Conflict of interest}: The author states that there is no conflict of
interest.


\begin{thebibliography}{99}
\bibitem{And00} M. T. Anderson, \emph{On the structure of solutions to the static vacuum Einstein
equations}, Ann. Henri Poincaré 1 (2000), no. 6, 977--1005.
\bibitem{Bar86} R. Bartnik, \emph{The mass of an asymptotically flat manifold}, Comm. Pure Appl.
Math. 39 (1986), no. 5, 661--693.
\bibitem{Bes87} A. L. Besse, \emph{Einstein Manifolds}, Ergebnisse der Mathematik und ihrer
Grenzgebiete (3) 10, Springer-Verlag, Berlin, 1987.
\bibitem{BiqOz26} O. Biquard and T. Ozuch, \emph{Instability of conformally Kähler, Einstein
metrics}, J. Differential Geom. 133 (2026), no. 3, 335--346. \url{https://doi.org/10.4310/jdg/1779980240}
\bibitem{CD03} \nv{P. T. Chru\'sciel and E. Delay, \emph{On mapping properties of the general
relativistic constraints operator in weighted function spaces, with applications}, M\'em. Soc.
Math. Fr. (N.S.) 94 (2003).}
\bibitem{Cor00} \nv{J. Corvino, \emph{Scalar curvature deformation and a gluing construction for
the Einstein constraint equations}, Comm. Math. Phys. 214 (2000), no. 1, 137--189.}
\bibitem{DK24} M. Dahl and K. Kröncke, \emph{Local and global scalar curvature rigidity of Einstein
manifolds}, Math. Ann. 388 (2024), no. 1, 453--510. \url{https://doi.org/10.1007/s00208-022-02521-6}
\bibitem{Del26} E. Delay, \emph{Instability of Riemannian Schwarzschild--Tangherlini metrics},
Lett. Math. Phys. 116 (2026), Paper No. 102.
\bibitem{Del12} \nv{E. Delay, \emph{Smooth compactly supported solutions of some underdetermined
elliptic PDE, with gluing applications}, Comm. Partial Differential Equations 37 (2012), no. 10,
1689--1716.}
\bibitem{DelTT} E. Delay, \emph{Une machine à tenseurs TT sur les variétés d'Einstein}, C. R. Math.
Acad. Sci. Paris 361 (2023), 495--506. \url{https://doi.org/10.5802/crmath.404}
\bibitem{GL93} \nv{R. Gregory and R. Laflamme, \emph{Black strings and $p$-branes are unstable},
Phys. Rev. Lett. 70 (1993), 2837--2840.}
\bibitem{GPY82} D. J. Gross, M. J. Perry and L. G. Yaffe, \emph{Instability of flat space at
finite temperature}, Phys. Rev. D 25 (1982), no. 2, 330--355.
\bibitem{HHM04} \nv{T. Hausel, E. Hunsicker and R. Mazzeo, \emph{Hodge cohomology of gravitational
instantons}, Duke Math. J. 122 (2004), no. 3, 485--548.}
\bibitem{LM85} R. B. Lockhart and R. C. McOwen, \emph{Elliptic differential operators on noncompact
manifolds}, Ann. Sc. Norm. Super. Pisa Cl. Sci. (4) 12 (1985), no. 3, 409--447.
\bibitem{Min09} \nv{V. Minerbe, \emph{Weighted Sobolev inequalities and Ricci flat manifolds},
Geom. Funct. Anal. 18 (2009), no. 5, 1696--1749.}
\bibitem{ONe83} B. O'Neill, \emph{Semi-Riemannian Geometry. With Applications to Relativity},
Pure and Applied Mathematics 103, Academic Press, New York, 1983.
\bibitem{Ozu26} T. Ozuch, \emph{Stability of gravitational instantons with a bounded Killing
vector field}, arXiv:2608.31008 (2026).
\bibitem{Pre00} \nv{T. Prestidge, \emph{Dynamic and thermodynamic stability and negative modes in
Schwarzschild--anti-de Sitter}, Phys. Rev. D 61 (2000), 084002.}
\bibitem{Rea01} \nv{H. S. Reall, \emph{Classical and thermodynamic stability of black branes},
Phys. Rev. D 64 (2001), 044005.}
\end{thebibliography}
\end{document}